\documentclass{article}

\usepackage{amsmath,amssymb,latexsym}
\usepackage{amsfonts}
\usepackage{amsthm}
\usepackage{amssymb}
\usepackage[margin=1in]{geometry}
\usepackage{enumerate}
\usepackage{graphicx}
\usepackage{geometry}
\usepackage{color}
\usepackage{csquotes}
\usepackage{dsfont}
\usepackage{authblk}
\usepackage{comment}

\newtheorem{theorem}{Theorem}[section]

\newtheorem{lemma}[theorem]{Lemma}
\newtheorem{corollary}[theorem]{Corollary}

\newtheorem{definition}[theorem]{Definition}
\newtheorem{example}[theorem]{Example}
\newtheorem{problem}[theorem]{Problem}

\newcommand{\R}{\mathbb{R}}
\newcommand{\I}{\mathcal{I}}
\newcommand{\F}{\mathbb{F}}

\newcommand{\mat}{$\mathcal{M}=(E,\mathcal{I})\text{ }$}

\def\a{\mathbf a}
\def\H{\mathbf H}
\def\C{\mathbf C}
\def\G{\mathbf G}
\def\mM{\mathbf M}
\def\P{\mathbf P}
\def\R{\mathbb R}
\def\s{\sigma_q}
\def\vs{\mathbf s}

\def\x{\mathbf x}
\def\y{\mathbf y}
\def\z{\mathbf z}

\begin{document}

\title{The Truncated \& Supplemented Pascal Matrix and Applications}
\author[1]{M. Hua\thanks{For correspondence regarding this paper, email mikwa@umich.edu. Support from the National Science 
Foundation under grants: 0901145, 1160720, 1104696 and the American Mathematical Society is gratefully acknowledged.}}
\author[2]{S. B. Damelin\thanks{damelin@umich.edu.}}
\author[1]{J. Sun\thanks{jeffjeff@umich.edu.}}
\author[3]{M. Yu\thanks{ming.yu@anu.edu.au.}}
\affil[1]{Department of Mathematics, University of Michigan.}
\affil[2]{Mathematical Reviews, The American Mathematical Society.}
\affil[3]{Australian National University.}

\date{}

\maketitle

\begin{abstract}
In this paper, we introduce the $k\times n$ (with $k\leq n$) truncated, supplemented Pascal matrix which has the property that any $k$ columns form a linearly independent set.  This property is also present in Reed-Solomon codes; however, Reed-Solomon codes are completely dense, whereas the truncated, supplemented Pascal matrix has multiple zeros.  If the maximal-distance separable code conjecture is correct, then our matrix has the maximal number of columns (with the aformentioned property) that the conjecture allows.  This matrix has applications in coding, network coding, and matroid theory.

\end{abstract}

\section{Introduction}
Finite field linear algebra is an import branch of linear algebra. Instead of using the infinite field $\R$, it uses linearly independent vectors consisting of a finite number of elements, which can be represented by a finite number of bits. It has thus motivated many practical coding techniques, such as Reed-Solomon codes \cite{RS} and linear network coding \cite{li2003linear,ho2006random}. It is also closely related to structural matroid theory \cite{oxley2006matroid} through matroid representability \cite{oxley2006matroid,oxley1996inequivalent,el2010index,yu2014deterministic}.

One of the most important problems in finite field linear algebra is finding the size of the largest set of vectors over a $k$-dimensional finite field such that every subset of $k$ vectors is linearly independent \cite{ball1, ball2}. From a matrix perspective, the problem is described as:
\begin{problem}
	Consider a finite field $\F_q$, where $q=p^h$, for $p$ a prime and $h$ a nonnegative integer. Given a positive integer $k$, what is the largest integer $n$, such that there exists a $k\times n$ matrix $\H$ over $\F_q$, in which every set of $k$ columns is linearly independent?
\end{problem}

Such a matrix, upon its existence, could be the generator matrix of an $[n, k]$ maximum-distance-separable (MDS) code \cite{costello2004error}, which can correct up to $d=n-k$ bits of erasures or $t=d/2$ bits of errors. We will thus refer to $\H $ as an MDS matrix. Its existence also determines the representability of uniform matroids, which we will discuss in detail in Section \ref{sec:matroid}. The maximal value of $n$, according to the MDS conjecture, is $q+1$, unless $q = 2^h$ and $k = 3$ or $k = q - 1$, in which case $n \leq q + 2$. This conjecture has been recently proved for any $q=p$ by Ball \cite{ball1,ball2}. But a complete proof of it remains open. 

Therefore, it is crucial to understand the construction of $k\times(q+1)$ MDS matrices. In coding theory literature, many construction algorithms have been proposed to meet certain coding requirements. However, their computational complexity is not necessarily satisfactory. On one hand, multiplications and additions over large finite field are required in the matrix construction. On the other hand, the resultant MDS matrix may have a low sparsity (or high density), which is measured by the number of zeros in the matrix. For example, Reed-Solomon codes have no zeros in its generator matrix. A low sparsity can be translated into higher encoding and decoding complexity.  It is an open question regarding how these algorithms can be generalized and provide new insights into related fields, such as network coding theory and matroid theory.

In this paper, we investigate the above problems by first proposing in Section \ref{sec:def} a new type of MDS matrix called a \emph{supplemented Pascal matrix}. A supplemented Pascal matrix can be generated by additions and, in particular, without multiplications. It also has guaranteed number of zero entries for high sparsity.  We will prove that a supplemented Pascal matrix is an MDS matrix in Section \ref{sec:proof}. We will then extend our results into a general code construction framework in section~\ref{sec:coding}, and then discuss its applications to network coding theory and matroid theory in sections \ref{sec:NC} and \ref{sec:matroid}, respectively.

\section{Definitions}\label{sec:def}
For clarity we should first label the elements of a finite field. Henceforth, let $p$ be a prime and $h$ be a nonnegative integer.  A finite field $\F_q$ contains $q=p^h$ elements, each represented by a polynomial $g(x)=\sum_{i=0}^{h-1}\beta_ix^i$, whose coefficients are $\{\beta_i\}_{i=0}^{h-1}\in[0,p-1]$. Substituting $x=p$ to a different $g(x)$ will yield a different value between 0 and $q-1$, which is an intuitive index of the corresponding element. Specifically, we define a index function $\s(n)$:
\begin{definition}
For any integer $n\in[0,q-1]$, $\s(n)$ is the element of $\F_q$ whose polynomial coefficients satisfy $\sum_{i=0}^{h-1}{\beta_ip^i}=n$.
\end{definition}
For example, given $q=2^3$, we have $\s(0)=0$, $\s(1)$=1, and $\s(5)=x^2+1$.

Based on $\s(n)$, we define a finite field binomial polynomial $f_m(n)$:
\begin{equation}
f_m(n)= \begin{cases}
1 = [\s(n)]^m,& m=0\\
\prod_{i=1}^{m}\frac{\s(n)-\s(i-1)}{\s(i)},&m>0
\end{cases}
\end{equation}
where $\{m,n\}\in[0,q-1]$ are non-negative integers. Intuitively, $f_m(n)$ is a polynomial of $\s(n)$ of degree $m$.

Based on $f_m(n)$, we introduce the key matrix in this paper, called the \emph{Pascal matrix}:

\begin{definition}
The upper-triangular Pascal matrix $\P_q$ over $\F_q$ is a $q \times q$ matrix with its element $\P_q(m,n) = f_m(n)$:
\begin{equation}
\P_q=\left[
\begin{array}{cccc}
f_0(0)&f_0(1)&\cdots&f_0(q-1)\\
f_1(0)&f_1(1)&\cdots&f_1(q-1)\\
\vdots&\ddots&\ddots&\vdots\\
f_{q-1}(0)&f_{q-1}(1)&\cdots&f_{q-1}(q-1)\\
\end{array}
\right],
\end{equation}
For brevity, we call the upper-triangular Pascal matrix the Pascal matrix.
\end{definition}

Note that the matrix index starts from 0. For example, when $q=2^2=4$, we have:

\begin{example}
$$
	\P_4=\begin{bmatrix}
	1 &1 &1 &1\\
	0 &1 &x &x+1\\
	0 &0 &1 &x+1\\
	0 &0 &0 &1\\
	\end{bmatrix}
$$	
\end{example}

Our considered matrix $\P_q$ is named after Pascal because its entries are binomial coefficients, which is the same as traditional Pascal matrix, except that the field applied is $\F_q$ and $\mathbb Z_{\geq0}$, respectively. Indeed, when $q=p$, $\P_p$ is equal to the traditional Pascal matrix modulo-$p$. For example, when $q=p=5$:
\begin{example}
$$ \P_{5,\mathrm{traditional}}=
	\begin{bmatrix}
	1 &1 &1 &1 &1\\
	0 &1 &2 &3 &4\\
	0 &0 &1 &3 &\textbf{6}\\
	0 &0 &0 &1 &4\\
	0 &0 &0 &0 &1\\
	\end{bmatrix} \mathrm{~~v.s.~~} \P_5=\begin{bmatrix}
	1 &1 &1 &1 &1\\
	0 &1 &2 &3 &4\\
	0 &0 &1 &3 &\textbf{1}\\
	0 &0 &0 &1 &4\\
	0 &0 &0 &0 &1\\
	\end{bmatrix}
$$
\end{example}
Indeed, the construction of the Pascal matrix under $\F_p$ shares the same additive formula as the traditional Pascal matrix. Explicitly, $\P_p(m,n) = \P_p(m-1,m-1) + \P_p(m, n-1)$ for every pair of $\{m, n\}\in[1, q-1]$ (note that addition is Mod-$p$).  This idea appears in section 4.2.

\begin{definition}
The truncated Pascal matrix $\P_{q,k}$ is the Pascal matrix $\P_q$ truncated to the first $k$ rows.
\end{definition}

\begin{example}
$$
	\P_{5,2}=\begin{bmatrix}
	1 &1 &1 &1\\
	0 &1 &2 &3\\
   \end{bmatrix}
$$
\end{example}

\begin{definition}
	A \emph{supplemented Pascal matrix}, denoted by $\H_{q,k}$, is a truncated Pascal matrix $\P_{q,k}$ appended with a column vector $\vs_k$, which has a one in the bottom entry and zeroes everywhere else:

\begin{equation}\label{eq:H_def}
\H_{q,k}=\left[~~~~~~U_{q,k}~~~~~~
\middle|
\begin{matrix}~0~\\
\vdots\\
~0~\\
~1~\end{matrix}\right]
\end{equation}
\end{definition}

\begin{example}
$$
	\H_{5,2}=\begin{bmatrix}
	1 &1 &1 &1 &0\\
	0 &1 &2 &3 &1\\
   \end{bmatrix}
$$	
\end{example}

Our supplemented Pascal matrix has a desirable property, namely:
\begin{theorem}\label{theo:H}
Any $k$ columns of $\H_{q,k}$ are linearly independent.
\end{theorem}

\section{Proof of Theorem \ref{theo:H}} \label{sec:proof}
We will first prove the following property of $\P_{q,k}$, and then prove that $\H_{q,k}$ preserves this property.

\begin{lemma}[Truncation Lemma]
\label{L1}
Any $k$ columns of $\P_{q,k}$ are linearly independent.
\end{lemma}
\begin{proof}
To prove it, we first note that $\P_q$ (and thus $\P_{q,k}$) has two important properties: ($m$ begins at 0.)
\begin{enumerate}
\item All the entries in the $m$-th row are defined by the same polynomial $f_m(n)$, which has a degree of $m$;
\item This polynomial has $m$ roots, which are $\{\s(n)\}_{n=0}^{m-1}$. Consequently, the first $m$ entries of the $m$-th row are all zeros.
\end{enumerate}

Given a truncated Pascal matrix $\P_{q,k}$, our hypothesis (to be disproved) is that there exist $k$ distinctive values of $n$, say $\{n_0,n_1,\cdots,n_{k-1}\}$, such that columns $\{n_0,n_1,\cdots,n_{k-1}\}$ of $\P_{q,k}$ constitute a linearly dependent set. In other words, if our hypothesis is valid, then there exists an $k\times k$ sub-matrix $\mM$ of $\P_{q,k}$:
\begin{equation}
\mathbf M=\left[
\begin{array}{cccc}
f_0(n_0)&f_1(n_1)&\cdots&f_1(n_{k-1})\\
f_1(n_0)&f_2(n_1)&\cdots&f_2(n_{k-1})\\
\vdots&\ddots&\ddots&\vdots\\
f_{k-1}(n_0)&f_{k-1}(n_1)&\cdots&f_{k-1}(n_{k-1})\\
\end{array}
\right],
\end{equation}
whose rank is smaller than $k$.

If this is the case, then there must exist a length-$k$ non-zero vector $\a\in\F_q^k$ such that $\a\times\mM=\z$, where $\z$ is an all-zero vector of length $k$:
$${[\alpha_0,\alpha_2,\cdots,\alpha_{k-1}]}_{\mathbf a}\times \mM=[0,0,\cdots,0]_{\z}$$

Recall that the $m$-th row of $\P_{q,k}$ (and thus $\mM$) is defined by $f_m(n)$. Correspondingly, $\z$  is defined by: $$f'(n)\triangleq\sum_{m=0}^{k-1}\alpha_mf_m(n),$$
where $\mathbf z(k)=f'(n_k)=0$ for all $m\in[0,k-1]$. We also note that the degree of $f'(n)$ is at most $k-1$, because the highest degree of its summands is the degree of $f_{k-1}(n)$ with a value of $k-1$.

Therefore, if our hypothesis is valid, i.e., if columns-$\{n_0,n_1,\cdots,n_{k-1}\}$ of $U_{q,k}$ constitute a linearly dependent set, then we will obtain a polynomial $f'(n)$ such that:
\begin{itemize}
\item Its degree is at most $k-1$;
\item It has $k$ roots, whose values are $\{\s(n_0),\s(n_1),\cdots,\s(n_{k-1})\}$.
\end{itemize}

However, with a degree of at most $k-1$, $f'(n)$ can only have at most $k-1$ roots unless $f'(n)=0$, which is not the case because $\mathbf a$ is non-zero. Hence, $f'(n)$ does not exist, and thus our hypothesis is invalid. Therefore, every $k$ columns of $U_{q,k}$ are linearly independent. \qed(lemma)

Since $\H_{q,k}$ is constructed by appending $\vs_k$ to $\P_{q,k}$, to prove Theorem \ref{theo:H} we only need to prove that any $k-1$ columns of $\P_{q,k}$ and $\vs_k$ together never constitute a linearly dependent set. To see this, we can simply use $\vs_k$ to linearly cancel the first $q$ entries in the last row of $\H_{q,k}$. This will transform $\H_{q,k}$ from (\ref{eq:H_def}) into:

\begin{equation}
\H_{q,k}' = \left[
\begin{array}{ccc|c}
~ & ~        & ~  & 0 \\
~ &\P_{q,k-1}& ~  & 0 \\
~ & ~        & ~  & \vdots \\
0 & \cdots   & 0  & 1
\end{array}
\right]
\end{equation}
which indicates that $\vs_k$ is orthogonal to all the other columns of $\H_{q,k}'$. Then, by applying the truncation lemma to $\P_{q,k-1}$, we know that every $k-1$ out of the first $q$ columns of $\H_{q,k}'$ are linearly independent. Adding $\vs_k$ to them will yield a linearly independent set of $k$. Theorem \ref{theo:H} is thus proved.
\end{proof}
\section{Applications}\label{sec:app}
\subsection{Coding Theory}\label{sec:coding}

The truncation lemma can be immediately generalized to any appropriately defined $k\times n$ matrix that satisfies: 1) $n\leq q$, and 2) the $m$-th ($m\in[0,k-1]$) row is defined by a polynomial of degree $m$. For example, by setting $f_m(n)=\s(n)^{m-1}$, we can obtain a $k\times n$ matrix under $\F_q$ such that every set of $k$ columns is a linearly independent set. Indeed, this matrix is the generator matrix $\G$ of a $(n,k)$ Reed-Solomon code:

\[\begin{bmatrix}
  	\s(1)^0 & \s(2)^0 &\cdots & \s(n)^0\\
  	\s(1) & \s(2) &\cdots & \s(n)\\
  	\s(1)^2 & \s(2)^2 &\cdots & \s(n)^2\\
  	\vdots &\vdots &\ddots &\vdots\\
  	\s(1)^{k-1} & \s(2)^{k-1} &\cdots & \s(n)^{k-1}
  \end{bmatrix}\]

Then by appending $\vs_k$, we can obtain a $[n+1,k]$ Reed-Solomon code. Therefore, our polynomial approach is a general approach of constructing \emph{non-trivial} $[n,k]$ MDS codes. It also indicates that the maximum length of any MDS code is at least $q+1$ for any $k\leqslant q$. This result well resonates the MDS conjecture \cite{ball1,ball2}.

Among all the possible constructions, the supplemented Pascal matrix $\H_{q,k}$ enjoys a high sparsity, which is the number of zeros in the matrix. Higher sparsity is advantageous, because it generally leads to easier encoding/decoding. However, the sparsity has an upper bound. In the following lemma, we will prove that  $\H_{q,k}$  approximates this bound with a factor of $\frac{1}{2}$:

\begin{lemma}[Matrix Sparsity]
The number of zeros in the supplemented Pascal matrix  $\H_{q,k}$ is $\frac{1}{2}$ of the maximum sparsity of any $(n,k)$ code.
\end{lemma}
\begin{proof}
Since any $k\times k$ sub-matrix of $\G$ has a rank of $k$, there is no all-zero row in this matrix. Hence, there is at most $k-1$ zeros in each row of $\G$, and at most $k^2+k$ zeros in total. Recall that in $\H_{q,k}$ the $m$-th row ($m\in[0,k-1]$) has $m$ zeros. The total number of zeros is $\frac{k^2-k}{2}$, which is half of the maximum.
\end{proof}

\subsection{Network Coding Theory}\label{sec:NC}
Network coding (NC) is a class of packet-based coding techniques. Consider a block of $K\geq 1$ data packets $\{\x_k\}_{k=0}^{K-1}$, each containing $L$ bits of information. NC treats these data packets as $K$ variables, and sends in the $u$-th ($u\in[0,+\infty]$) transmission a linear combination $\y_u$ of all of them:

\begin{equation}
\y_u=\sum_{k=0}^{K-1}\alpha_{k,u}\x_k,
\end{equation}
where coefficients $\{\alpha_k\}_{k=0}^{K-1}$ are elements of a finite field $\F_q$.

Ideally, NC is able to allow any receiver that has received any $K$ coded packets to decode all the $K$ data packets by solving a set of $K$ linear equations. To this end, the associated coefficient matrix $\C$, where

\begin{equation}
\C=\left[
\begin{array}{cccc}
\alpha_{0,0}&\alpha_{0,1}&\cdots&\cdots\\
\alpha_{1,0}&\alpha_{1,1}&\cdots&\cdots\\
\vdots&\vdots&\ddots&\vdots\\
\alpha_{K-1,0}&\alpha_{K-1,1}&\cdots&\cdots\\
\end{array}
\right],
\end{equation}
must satisfy that every set of $K$ columns of it is a linearly independent set. Once this condition is met, NC is able to achieve the optimal throughput in wireless broadcast scenarios \cite{yu2014deterministic}.

However, it is highly non-trivial to meet this condition, which hinders the implementation of NC. First, to guarantee the linear independence, the sender either chooses coefficients randomly from a sufficiently large $\F_q$ \cite{lucani2009random,heide2009network} or regularly collect receiver feedback to make online coding decisions \cite{fragouli2007feedback}. While large $\F_q$  incurs heavy computational loads, collecting feedback could be expensive or even impossible in certain circumstances, such as time-division-duplex satellite communications \cite{lucani2009random}. Second, to enable the decoding, coding coefficients must be attached to each coded packet, which constitute $\lceil K\log_2q \rceil$ bits of overhead in each transmission. When $q$ is large and $L$ is small, the throughput loss due to the overhead may overwhelm all the other benefits of NC.

These practical shortages of NC can be easily overcome by the proposed supplemented Pascal matrix. By choosing a sufficiently large $p$ and let $\C = \H_{p,K}$, we obtain an NC that is both computational friendly (only Mod-$p$ operations) and feedback-free. Moreover, for the receivers to retrieve the coding coefficients, the sender only needs to attach the index $u$ to the $u$-th packet, rather than attaching the complete coefficients. Furthermore, the additive formula for Pascal matrix may enable efficient progressive coding/decoding algorithms, which could be our future research direction.

\subsection{Matroid Theory}\label{sec:matroid}
A matroid \mat is a finite collection of elements called the ground set, $E$, paired with its comprehensive set of independent subsets, $\I$.  A uniform matroid $U_n^k$ has $|E|=n$ and the property that {\bf any} size $k$ subset of $E$ is an element of $\I$ and {\bf no} size $(k+1)$ subset is in $\I$. $U_n^k$ is called $q$-representable if there is a $k\times n$ matrix such that every $k$ columns of it are linearly independent under $\F_q$.  
 
\begin{corollary}[Representability of Uniform Matroid]
Any uniform matroid $U_{n}^k$ that satisfies $n\leqslant q+1$ is $q$-representable by any $n$ columns of $\H_{q,k}$.\qed
\end{corollary} 

The statement: \textit{Any uniform matroid $U_{n}^k$ that satisfies $n\leqslant q+1$ is $q$-representable} is known \cite{oxley2006matroid, ball3, RS}; one can obtain another construction from Reed-Solomon codes.  $\H_{q,k}$ is just another, sparse example.

\section{Conclusion}

In this paper, we proposed the supplemented Pascal matrix, whose first $k$ rows is an MDS matrix under $\F_q$ for any prime power $q$ and positive integer $k\leqslant q$. Our construction can be potentially generalized to a framework that enables low-complexity MDS code constructions and encoding/decoding as well. Our matrix can overcome some practical shortages of network coding and, thus, enables high-performance wireless network coded packet broadcast. Our matrix resonates with existing results on the representability of uniform matroids, while also providing new insights into this topic. In the future, we intend to study Pascal-based network coding algorithms. We are also interested in applying our results to other fields such as projective geometry and graph theory.


\end{document}